\documentclass[11pt]{amsart}
\usepackage{amsmath,amssymb,mathrsfs,stmaryrd, color,geometry}
\newtheorem{theorem}{Theorem}
\newtheorem{proposition}[theorem]{Proposition}
\newtheorem{definition}[theorem]{Definition}

\newtheorem{lemma}[theorem]{Lemma}

\newtheorem{remark}[theorem]{Remark}
\newcommand{\R}{\mathbb R}

\numberwithin{theorem}{section}
\numberwithin{equation}{section}

\begin{document} 

\title[Angular momentum and center of mass]{Conserved Quantities of harmonic asymptotic
 \\ 
initial data sets }
\author{Po-Ning Chen and Mu-Tao Wang}
\begin{abstract} 
In the first half of  this article, we survey the new quasi-local and total angular momentum and center of mass defined in \cite{Chen-Wang-Yau3} and summarize the important properties of these definitions. To compute these conserved quantities involves solving a nonlinear PDE system (the optimal isometric embedding equation), which is rather difficult in general. We found a large family of initial data sets on which such a calculation can be carried out effectively. These are initial data sets of {\it harmonic asymptotics}, first proposed by  Corvino and Schoen to solve the full vacuum constraint equation. In the second half of this article, the new total angular momentum and center of mass for these initial data sets are computed explicitly. 
\end{abstract}
\address{Department of Mathematics \\ Columbia University \\ 2990 Broadway \\ New York, NY 10027}
\thanks{ P.-N. Chen is supported by NSF grant DMS-1308164 and M.-T. Wang was supported by NSF grants DMS-1105483 and DMS-1405152. Both authors
would like to thank Justin Corvino, Lan-Hsuang Huang, Richard Schoen, and Shing-Tung Yau. In particular, they learned the values of the ADM angular momentum
and BORT center of mass for harmonic asymptotic initial data sets from Huang. }
\date{\today}
\maketitle 
\section{Introduction}
By Noether's principle, any continuous symmetry of the action of a physical system corresponds to a conserved quantity. In special relativity, integrating the energy-momentum tensor of matter density against Killing fields in $\mathbb{R}^{3,1}$ gives the well-defined notion of energy momentum, angular momentum, and center of mass. In attempt to generalize these concepts to general relativity, one encounters two major difficulties. Firstly,  gravitation does not have mass density. Secondly, there is no symmetry in a general spacetime. As such, most studies of conserved quantities are restricted to isolated systems on which asymptotically flat coordinates exist at infinity. However, it has been conjectured \cite{Penrose} for decades that a quasi-local description of conserved quantities should exist, at least for energy-momentum and angular momentum. These are notions attached to a spacelike 2-surface $\Sigma$ in spacetime. In \cite{Wang-Yau1}, a new definition of quasi-local energy-momentum and quasi-local mass was proposed using isometric embeddings of the 2-surface into $\mathbb{R}^{3,1}$. The expressions originated from the boundary term in the Hamilton--Jacobi analysis of gravitation action \cite{Brown-York2, hh}. To each pair of $(X, T_0)$ where $X:\Sigma\hookrightarrow \mathbb{R}^{3,1}$ is an isometric embedding and $T_0$ is a future time-like unit Killing field in $\mathbb{R}^{3,1}$, a canonical gauge is chosen and a  quasi-local energy is assigned. The quasi-local mass is obtained by minimizing the quasi-local energy seen among admissible $(X, T_0)$. For each critical point of the quasi-local energy, the isometric embedding and the canonical gauge were used to transplant Killing fields in $\mathbb{R}^{3,1}$ back to the surface of interest in the physical spacetime. In particular, this defines quasi-local angular momentum and quasi-local center of mass with respect to rotation Killing fields and boost Killing fields. We refer to \cite{Szabados} for earlier work on the definition of quasi-local angular momentum, notably \cite{Penrose2}. This new proposal is further applied to study asymptotically flat spacetime and to define new total conserved quantities on asymptotically flat \cite{Chen-Wang-Yau3} and asymptotically hyperbolic \cite{Chen-Wang-Yau4} initial data sets.  The new definition satisfies highly desirable properties, and captures the dynamics of the Einstein equation. 

In studying the dynamics of the Einstein equation, an important problem is the construction of initial data sets. The initial data consists of $(M, g, k)$ where $g$ is the induced metric and $k$ is the second fundamental form of the manifold $M$ in a given spacetime. For vacuum spacetimes, the initial data sets satisfy the vacuum constraint equation. 
\begin{definition}
We said that $(M,g,k)$ satisfies the vacuum constraint equation if 
\[  
\begin{split}
R(g)+ (tr_g k)^2- |k|_g^2 = & 0 \\
div_g(k - (tr_g k )g) = & 0.
\end{split}
\] 
Here $R(g)$ is the scalar curvature of the metric $g$.
\end{definition}
For time-symmetric initial data with $k=0$, the vacuum constraint equation reduces to the scalar curvature equation $R(g)=0$, a semi-linear elliptic equation for the conformal factor of a conformal metric $g$. Such initial data sets can be constructed effectively using the quasi-spherical gauge method introduced by Bartnik \cite{Bartnik}. This construction plays an important role in the proof of the positivity of quasi-local mass  \cite{Liu-Yau,Shi-Tam}.

For initial data sets with non-trivial $k$, it is much harder to reduce the under-determined vacuum constraint equation to a well-posed equation. One of the methods to construct such initial data sets is the  conformal method by York \cite{York}. The conformal method is particularly useful for numerical general relativity, for example, to construct initial data sets representing binary black holes. A more recent method is a gluing construction based on the harmonic asymptotic ansatz introduced in \cite{Corvino-Schoen} by Corvino and Schoen. The ansatz provides a way to reduce the  vacuum constraint equation to an elliptic systems with four unknowns and four equations, see \eqref{vacuum_constraint}. More importantly, the  Arnowitt--Deser--Misner (ADM) energy, linear momentum, angular momentum and center of mass can be read off from the asymptotics. In \cite{Huang-Wang-Schoen}, the harmonic asymptotic initial data sets are used  to construct counterexamples to the conjectured angular momentum-mass inequality. In this article, we compute the new (CWY) total angular momentum and center of mass defined in \cite{Chen-Wang-Yau3} for harmonic asymptotic initial data sets.

In the first half of the article, we review the definition of the quasi-local energy-momentum in \cite{Wang-Yau2} and the quasi-local angular momentum and center of mass in \cite{Chen-Wang-Yau3} and describe how to define the total conserved quantities for an asymptotically flat initial data using the newly defined quasi-local conserved quantities. The important properties of the new quasi-local and total conserved quantities are summarized in Section \ref{sec_list_properties}. In the second half of the article, we compute the CWY total conserved quantities for harmonic asymptotic initial data. The computation provides a large family of examples of non-vanishing linear momentum on which the procedure outlined in Section \ref{sec_outline_total} is carried out.

\section{ADM energy-momentum, angular momentum and center of mass} \label{sec_ADM}
In this section, we review the ADM energy-momentum, ADM angular momentum \cite{Arnowitt-Deser-Misner} and the Beig--\'O Murchadha--Regge--Teitelboim (BORT) center of mass \cite{Beig-Omurchadha,Regge-Teitelboim} for asymptotically flat initial data sets.  

\begin{definition} $(M, g, k)$ is said to be asymptotically flat if there exists a compact subset $K\subset M$ such that $M\backslash K$ is diffeomorphic to a union of complements of balls in $\R^3$ (called ends). The diffeomorphism provides an asymptotically flat coordinate 
system on each end such that

\[ g_{ij}=\delta_{ij} +a_{ij} \text{ with }
a_{ij}=O({r}^{\alpha}), \partial_k(a_{ij})=O(r^{\alpha-1}), \partial_l\partial_k(a_{ij})=O(r^{\alpha-2}),\] and  \[ k_{ij}=O(r^{\beta}),
\partial_k(k_{ij})=O(r^{\beta-1})\] on each end of $M\backslash K$.  
\end{definition}

The ADM energy-momentum of asymptotically flat initial data sets is defined as follows.

\begin{definition} 
Let $(M, g_{ij}, k_{ij})$ be an asymptotically flat initial data. The ADM  energy momentum vector $(E,P_k)$ of an end of $M$  is

 \[
\begin{split}
E=&\frac{1}{16\pi } \lim_{r\rightarrow \infty} \int_{\Sigma_r}\sum_{i,j}(\partial_j g_{ij}-\partial_i g_{jj})\nu^i d\Sigma_r \\
P_k=& \frac{1}{8\pi }\lim_{r\rightarrow \infty} \int_{\Sigma_r}
\sum_{i}\pi_{ik}\nu^i d\Sigma_r, 
\end{split}
\] where  $\pi_{ik}= k_{ik}-g_{ik} k_j^{j}$. Here $\Sigma_r$ is a coordinate sphere of radius $r$ with respect to the asymptotically flat coordinate system on this end and $\nu$ is the unit normal of $\Sigma_r$ in $M$.  
\end{definition}
The ADM angular momentum and BORT center of mass are defined as follows. 
\begin{definition} 
Let $(M, g, k)$ be an asymptotically flat initial data set. The center of mass $C^i$ and the angular momentum $J_i $ with respect to $Y_{(i)}=\frac{\partial}{\partial x^i} \times \vec{x}$, for $i=1,2,3$, are defined by
 \[
\begin{split}
	C^{i}
	= & \frac{1}{ 16 \pi } \lim_{r\rightarrow \infty} \int_{\Sigma_r} \Big  [ x^i\sum_{k,j}\left(\partial_k g_{kj}-\partial_j g_{kk} \right)\nu^j -
	        \sum_k ( g_{ik}\nu^k -g_{kk} \nu^{i}) \Big ] \, d\Sigma_r \\
	J_i=& \frac{1}{ 8 \pi } \lim_{r\rightarrow \infty} \int_{\Sigma_r} \sum_{j,k} \pi_{jk} Y^j_{(i)} \nu^k \, d\Sigma_r.
\end{split}
\]
\end{definition}

In the definition of the angular momentum, $Y_{(i)}$ plays the role of an asymptotically Killing field, which ultimately depends on the asymptotically flat coordinate system.

The ADM angular momentum and BORT center of mass are more complicated than the ADM energy-momentum since the corresponding Killing fields are of higher order near spatial infinity. In particular, there exist hypersurfaces in the Minkowski spacetime with non-trivial total angular momentum and center of mass.  Explicit examples of such hypersurfaces can be found in \cite[Theorem 4.1 and 4.7]{Chen-Huang-Wang-Yau}. We describe the example for ADM angular momentum here.

Let $\tilde X^i$, $i=1,2,3$, be the three standard $-2$ eigenfunctions of the Laplacian (or coordinate functions of the standard embedding 
into $\R^{3}$) of the round metric on $S^2$. Consider the the hypersurface defined by $t=r^{\frac{1}{3}}\tilde X^1 (\tilde X^2)^3$ in the Minkowski spacetime. While the hypersurface has vanishing ADM energy-momentum and vanishing BORT center of mass,   
its ADM angular momentum is finite and non-zero. 

\section{Quasi-local energy-momentum, center of mass and angular momentum}
We recall the definition of the quasi-local energy-momentum defined in  \cite{Wang-Yau2} and the quasi-local angular momentum and center of mass defined in \cite{Chen-Wang-Yau3}.

Let $\Sigma$ be a closed embedded spacelike 2-surface in a spacetime with spacelike mean curvature vector $H$. The data used in the definition of quasi-local mass is the triple $(\sigma,|H|,\alpha_H)$ where $\sigma$ is the induced metric on $\Sigma$, $|H|$ is the norm of the mean curvature vector and $\alpha_H$ is the connection one form of the normal bundle with respect to the mean curvature vector
\[ \alpha_H(\cdot )=\langle \nabla^N_{(\cdot)}   \frac{J}{|H|}, \frac{H}{|H|}  \rangle  \]
where $J$ is the reflection of $H$ through the incoming  light cone in the normal bundle.

Given an isometric embedding $X:\Sigma\rightarrow \R^{3,1}$ and a constant future timelike unit vector $T_0\in \R^{3,1}$, we consider the projected embedding $\widehat{X}$ into the orthogonal complement of $T_0$, we denote the induced metric, the second fundamental form, and the mean curvature of the image by $\hat{\sigma}_{ab}$, $\hat{h}_{ab}$, and $\widehat{H}$, respectively.  

Let $\tau= -X \cdot T_0$. The quasi-local energy with respect to $(X, T_0)$ is (see \S 6.2 of \cite{Wang-Yau2})
\[\begin{split}&E(\Sigma, X, T_0)=\int_{\widehat{\Sigma}}\widehat{H}d{\widehat{\Sigma}}-\int_\Sigma \left[\sqrt{1+|\nabla\tau|^2}\cosh\theta|{H}|-\nabla\tau\cdot \nabla \theta -\alpha_H ( \nabla \tau) \right]d\Sigma,\end{split}\] where $\theta=\sinh^{-1}(\frac{-\Delta\tau}{|H|\sqrt{1+|\nabla\tau|^2}})$ and $\nabla$ and $\Delta$ are the gradient and Laplacian, respectively, with respect to $\sigma$.

In relativity, the energy of a moving particle depends on the observer, and the rest mass is the minimal energy seen among all observers. In order to define quasi-local mass, we minimize quasi-local energy $E(\Sigma, X, T_0)$ among all admissible pairs $(X, T_0)$ which are considered to be quasi-local observers. A critical point corresponds to an optimal isometric embedding which is defined as follows \cite{Wang-Yau2}:

\begin{definition}
Given a closed embedded spacelike 2-surface $\Sigma$ in $N$ with $(\sigma,|H|,\alpha_H)$, an optimal isometric embedding is an embedding $X:\Sigma\rightarrow \R^{3,1}$ such that the induced metric of the image surface in $\R^{3,1}$ is $\sigma$, and there exists a $T_0$ such that
$\tau=-\langle X, T_0\rangle$ satisfies
\begin{equation} \label{optimal}
 -(\widehat{H}\hat{\sigma}^{ab} -\hat{\sigma}^{ac} \hat{\sigma}^{bd} \hat{h}_{cd})\frac{\nabla_b\nabla_a \tau}{\sqrt{1+|\nabla\tau|^2}}+ div_\sigma (\frac{\nabla\tau}{\sqrt{1+|\nabla\tau|^2}} \cosh\theta|{H}|-\nabla\theta-\alpha_{H})=0.
\end{equation}
\end{definition}
Denote the norm of the mean curvature vector and the connection one-form in mean curvature gauge of the image of $X$ by $|H_0|$ and $\alpha_{H_0}$, respectively. The quasi-local energy is best described by $\rho $ and $j_a$ defined as follows: 
\[
\begin{split}
 \label{rho}\rho &= \frac{\sqrt{|H_0|^2 +\frac{(\Delta \tau)^2}{1+ |\nabla \tau|^2}} - \sqrt{|H|^2 +\frac{(\Delta \tau)^2}{1+ |\nabla \tau|^2}} }{ \sqrt{1+ |\nabla \tau|^2}}
\\ 
j_a & =\rho {\nabla_a \tau }- \nabla_a [ \sinh^{-1} (\frac{\rho\Delta \tau }{|H_0||H|})]-(\alpha_{H_0})_a + (\alpha_{H})_a.\end{split}
\]
In terms of these, the quasi-local energy is $\frac{1}{8\pi}\int_\Sigma (\rho+j_a\nabla^a\tau)$ and a pair of an embedding $X:\Sigma\hookrightarrow \mathbb{R}^{3,1}$ and an observer $T_0$ satisfies the optimal isometric embedding equation if $X$ is an isometric embedding and 
\[
div_{\sigma} j=0.
\]

We define the quasi-local angular momentum and center of mass for each optimal isometric embedding. Let $(x^0,x^1,x^2,x^3)$ be the standard coordinate of $\R^{3,1}$. We recall that $K$ is a rotational Killing field if $K$ is the image of $x^i \frac{\partial}{\partial x^j}-x^j \frac{\partial}{\partial x^i} $ under a Lorentz transformation. Similarly, $K$ is a boost Killing field if $K$ is the image of $x^0 \frac{\partial}{\partial x^i}+x^i \frac{\partial}{\partial x^0} $ under a Lorentz transformation. 

\begin{definition} 
The quasi-local conserved quantity of $\Sigma$ with respect to an optimal isometric embedding $(X, T_0)$ and a Killing field $K$ is 
\[
E(\Sigma, X, T_0, K)=\frac{(-1)}{8\pi} \int_\Sigma \left[ \langle K, T_0\rangle \rho+j(K^\top) \right]d\Sigma,
\]
where $K^\top$ is the tangential part of $K$ to $X(\Sigma)$. 

Suppose $T_0=A(\frac{\partial}{\partial x^0})$ for a Lorentz transformation $A$, then the quasi-local conserved quantities corresponding to $A(x^i\frac{\partial}{\partial x^j}-x^j\frac{\partial}{\partial x^i}), i<j$ are called the quasi-local angular momentum integral with respect
to $T_0$ and the quasi-local conserved quantities corresponding to $A(x^i\frac{\partial}{\partial x^0}+x^0 \frac{\partial}{\partial x^i}), i=1, 2, 3$ are called the quasi-local center of mass integral with respect to $T_0$. 
\end{definition}

\section{Total conserved quantities  for order $1$ data} \label{sec_outline_total}
In this section, we review the definition of total conserved quantities using the limits of quasi-local conserved quantities (see Section 6 of \cite{Chen-Wang-Yau3}).

First we recall the definition of asymptotically flat initial data set of order $1$.
\begin{definition}\label{order_one}
 $(M, g, k)$ is said to be \textit{asymptotically flat of order 1} if there is a compact subset $K$ of $M$ such that 
$ M \backslash K $ is diffeomorphic a union of ends. In the asymptotically flat coordinate system on each end, we have the following decay condition for $g$ and $k$.
\begin{equation*} 
\begin{split}
g_{ij} & = \delta_{ij}+ \frac{g_{ij}^{(-1)}}{r}+  \frac{g_{ij}^{(-2)}}{r^2}+ o(r^{-2})\\
k_{ij} & =   \frac{k_{ij}^{(-2)}}{r^2}   +  \frac{k_{ij}^{(-3)}}{r^3} +o(r^{-3}).
\end{split}
\end{equation*}
\end{definition}

 On each level set of $r$,  $\Sigma_r$, we can use $\{u^a\}_{a=1, 2}$ as coordinate system to express the geometric data we need in order to define the quasi-local conserved quantities: 
\[  
\begin{split}
\sigma_{ab} & = r^2 \tilde \sigma_{ab}+ r \sigma_{ab}^{(1)} + \sigma_{ab}^{(0)}+ o(1) \\
|H| & = \frac{2}{r}+ \frac{h^{(-2)}}{r^2}+\frac{h^{(-3)}}{r^3} + o(r^{-3}) \\
\alpha_H & = \frac{\alpha_H^{(-1)} }{r}+ \frac{\alpha_H^{(-2)} }{r^2} + o(r^{-2})
\end{split}
\] 
where $\tilde{\sigma}_{ab}$ is the standard metric on a unit round sphere $S^2$. $\sigma_{ab}^{(1)}$ and $ \sigma_{ab}^{(0)}$ are considered as symmetric 2-tensors on $S^2$, $h^{(-2)}$ and $h^{(-3)}$ are functions on $S^2$ and $\alpha_H^{(-1) }$ and  $\alpha_H^{(-2)}  $ are 1-forms on $S^2$. We shall see that all total conserved quantities on $(M, g, k)$ can be determined by these data.

Using the result of \cite{Chen-Wang-Yau1}, there is a unique family of isometric embeddings $X(r)$ and observers $T_0(r)$ which minimizes the quasi-local energy locally. 
The embedding $ X({r})= (X^0({r}), X^i({r})) $ has the following expansion
\begin{equation}
\begin{split}\label{optimal_order} 
X^0({r}) & = (X^0)^{(0)} + \frac{(X^0)^{(-1)}}{r} + o(r^{-1}) \\
X^i ({r})& = r \tilde X^i + (X^i)^{(0)} + \frac{(X^i)^{(-1)}}{r} + o(r^{-1}) \\
T_0 ({r})& = (a^0,a^i) + \frac{1}{r}(0, (a_i)^{(-1)}) + o(r^{-1}).
\end{split}
\end{equation}

Using these, we compute 
\[  
\begin{split}
|H_0| & = \frac{2}{r}+ \frac{h_0^{(-2)}}{r^2}+\frac{h_0^{(-3)}}{r^3} + o(r^{-3}) \\
\alpha_{H_0} & = \frac{\alpha_{H_0}^{(-1)} }{r}+ \frac{\alpha_{H_0}^{(-2)} }{r^2} + o(r^{-2})
\end{split}
\] 
and the expansion for $\rho$ and $j_a$ 
\[  
\begin{split}
\rho & =  \frac{\rho^{(-2)}}{r^2}+\frac{\rho^{(-3)}}{r^3} + o(r^{-3}) \\
j_a& = \frac{j_a^{(-1)} }{r}+ \frac{j_a^{(-2)} }{r^2} + o(r^{-2}).
\end{split}
\] 

The isometric embedding equation gives 

\[-\partial_a X^0\partial_b X^0+\sum_i \partial_a X^i\partial_b X^i=\sigma_{ab}.\]

By the expansions of $X^0$ and $X^i$ in equation \eqref{optimal_order} , we obtain the following two equations:

\begin{align}\label{isometric1}\partial_a \tilde{X}^i\partial_b (X^i)^{(0)}+\partial_a (X^i)^{(0)}\partial_b \tilde{X}^i= &(\sigma_{ab})^{(1)}\\
\label{isometric2}\partial_a \tilde{X}^i\partial_b (X^i)^{(-1)}+\partial_a (X^i)^{(-1)}\partial_b \tilde{X}^i+\sum_i \partial_a (X^i)^{(0)}\partial_b (X^i)^{(0)}=& (\sigma_{ab})^{(0)} +  \partial_a (X^0)^{(0)}\partial_b (X^0)^{(0)}.\end{align}
On the other hand, the optimal embedding equation reads $div_{\sigma} j =0$. From the expansion of metric and $j_a$, we have
${div_{\tilde \sigma}} j^{(-1)}=0$ which determines $(X^0)^{(0)}$ and $(a^0,a^i)$. The equation reads
 \begin{equation}\label{linearized_optimal}{div_{\tilde \sigma}}(\rho^{(-2)}\widetilde{\nabla}\tau^{(1)})-\frac{1}{4} \widetilde{\Delta} (\rho^{(-2)}\widetilde{\Delta} \tau^{(1)})-\frac{1}{2}\widetilde{\Delta}(\widetilde{\Delta}+2)(X^0)^{(0)}+{div_{\tilde \sigma}}( \alpha_H^{(-1)})=0\end{equation} 
where $ \tau^{(1)} = -\sum_i a^i \tilde X^i$. 

We recall the definitions in \cite{Chen-Wang-Yau3}.
\begin{definition}\label{total_conserved}
Suppose $ T_0({r})=A(r)(\frac{\partial}{\partial x^0})$ for a family of Lorentz transformation $A(r)$.  The CWY total center of mass integral is 
defined to be
\[
C^i  =  \lim_{r \to \infty} E(\Sigma_r, X(r), T_0(r), A({r})(x^i\frac{\partial}{\partial x^0}+x^0\frac{\partial}{\partial x^i})),\\
\] and the CWY total angular momentum integral is defined to be
\[
J_{i} =\lim_{r \to \infty} \epsilon_{ijk}E(\Sigma_r , X(r), T_0(r), A({r})(x^j\frac{\partial}{\partial x^k}-x^k\frac{\partial}{\partial x^j})) \]
where  $\Sigma_r$ are the coordinate spheres and $(X(r), T_0({r}))$ is the unique family of optimal isometric embeddings of $\Sigma_r$ such that $X(r)$ converges to a round sphere of radius $r$ in $\R^3$ when $r\rightarrow \infty$.
\end{definition}

\section{Properties of the CWY total conserved quantities} \label{sec_list_properties}
In this section, we summarize the properties of the CWY total conserved quantities.  
\begin{enumerate}
\item The definition depends only on the geometric data $(g, k)$ and the foliation of surfaces at infinity, and in particular does not depend on an asymptotically flat coordinate system or the existence of an asymptotically Killing field. 
\item All CWY total conserved quantities vanish on any spacelike hypersurface in the Minkowski spacetime, regardless of the asymptotic behavior. In fact, all the quasi-local conserved quantities vanishes for surfaces in the Minkowski spacetime since the definition uses comparison of physical Hamiltonian with reference Hamiltonian. As pointed out in Section \ref{sec_ADM}, this is not true for the ADM angular momentum and BORT center of mass. 

\item 
For the  ADM angular momentum and BORT center of mass, there are issues of finiteness since the rotational and boost Killing fields are of higher order.
On the contrary, the CWY total angular momentum and total center of mass are always finite on any vacuum asymptotically flat initial data set of order $1$.  See \cite[Theorem 7.4]{Chen-Wang-Yau3}.

\item The CWY  total angular momentum satisfies the conservation law. In particular, the CWY total angular momentum on any strictly spacelike hypersurface of the Kerr spacetime is the same. Namely, for any spacelike hypersurface in the Kerr spacetime defined by $t= f(r,u^a)$ for $r>>1$ where $f=o(r)$, its CWY total angular momentum is the same as the stationary hypersurface $t=0$. See \cite[Theorem 8.4]{Chen-Wang-Yau3}. 

\item Under the vacuum Einstein evolution of initial data sets, the CWY total angular momentum is conserved, $\partial_t J_i(t)=0$, and the CWY  total center of mass obeys the dynamical formula $\partial_t C^i(t)=\frac{P^i}{E}$ where $(E,P^i)$ is the ADM energy-momentum four vector. 

In \cite{Chen-Wang-Yau3}, the following theorem was proved:
\begin{theorem}  \label{cor_variation_center}
Let $(M, g, k)$ be a vacuum asymptotically flat initial data set of order $1$. Let $(M, g(t), k(t) )$ be the solution to the initial value problem $g(0)=g$ and $k(0)=k$ for the vacuum Einstein equation with lapse function $N=1+O(r^{-1})$ and shift vector $\gamma=\gamma^{(-1)} r^{-1}+O(r^{-2})$. Let $C^i(t)$ and $J_i(t)$ be the the total center of mass  and total angular momentum  of $(M, g(t), k(t))$ respectively. We have
\[
\begin{split}
\partial_t C^i (t)= &  \frac{P^i}{E},\\
\partial_t J_{i} (t) = &  0
\end{split} 
\]
for $i=1, 2, 3$ where $(E,P^i)$ is the ADM energy momentum of $(M, g, k)$.
\end{theorem}
In fact, the dynamical formula was proved for initial data sets with much weaker asymptotics. See \cite[Theorem 10.2]{Chen-Wang-Yau3}.
\item The CWY quasi-local angular momentum and center of mass can also be used to define total conserved quantities for asymptotically hyperbolic initial data sets  \cite{Chen-Wang-Yau4}. They can be defined using the limit of quasi-local conserved quantities in the same fashion since the data on the coordinate spheres admit expansion of the same form. The main difference is that the defined total conserved quantities is finite only for the foliation with vanishing linear momentum. Assuming the total energy-momentum is timelike, there exists a foliation with vanishing linear momentum by applying coordinate boost. The limit of quasi-local conserved quantities is evaluated for such foliation in terms of the expansion of the metric and second fundamental form. See \cite[Theorem 7.3]{Chen-Wang-Yau4}.

\end{enumerate}
\section{harmonic asymptotic initial data sets}
In the remaining part of the article, we consider the total conserved quantities for harmonic asymptotic initial data sets.  Recall that 
\begin{definition}
An initial data set  $(M,g,k)$ is said to be of harmonic asymptotics if it is asymptotically flat and on each end there is a function $u$ and a vector field $Y$ on $M$ such that 
\[
\begin{split}
g_{ij} = & u^4 \delta_{ij}\\
k_{ij}= & u^2 (Y_{i, j}+Y_{j,i}-\frac{Y_{k,k}}{2} \delta_{ij})
\end{split}
\]
 where $u$ tends to $1$ and $Y$ tends to $0$ at infinity. 
\end{definition}
For simplicity, let
\[(\mathfrak{L} Y)=Y_{i, j}+Y_{j,i}-(Y_{k,k}) \delta_{ij}.\]

For a harmonic asymptotic initial data, the vacuum constraint equation is reduced to the following elliptic system for $u$ and $Y_i$:
\begin{equation}\label{vacuum_constraint}
\begin{split}
8\Delta_\delta u= &u(-|\mathfrak{L} Y|^2+\frac{1}{2} (Tr(\mathfrak{L}Y))^2)\\
\Delta_\delta Y_i= &2u^{-1}u_i Tr(\mathfrak{L}Y) -4u^{-1} u_j(\mathfrak{L}Y)_i^j,
\end{split}
\end{equation}
where $\Delta_\delta$ is the Laplacian with respect to the flat metric $\delta_{ij}$ on $\R^3$.

Suppose $u$ and $Y_i$ admit expansions of the following form:
\begin{equation}\label{expansion_u_Y}
\begin{split} u&=1+\frac{A}{r}+\frac{u^{(-2)}}{r^2}+o(r^{-2})\\
Y_i&=\frac{B_i}{r}+\frac{ {Y_i}^{(-2)}}{r^2}+o(r^{-2}).\end{split}\end{equation}
We show that $u^{(-2)}$ and $ Y_i^{(-2)}$ can be determined explicitly by \eqref{vacuum_constraint} in the following lemma.
\begin{lemma}\label{harmonic_expansion}
Let $(M,g,k)$ be a harmonic asymptotic initial data satisfying the vacuum constraint equation where $u$ and $Y$ have expansion as in equation \eqref{expansion_u_Y}. Then
\[
\begin{split}u^{(-2)}=& c_i \tilde X^i -\frac{9}{64}\sum_i (B_i)^2+\frac{1}{64}\sum_{i, j} B_i B_j \tilde{X}^i \tilde{X}^j \\
 Y_i^{(-2)}=&d_{ij} \tilde X^j-\frac{5}{2} B_i+\frac{1}{2}(\sum_k B_k \tilde X^k)\tilde{X}^i.
\end{split}
\]
\end{lemma}
\begin{proof}
Expanding $\Delta_\delta$ into spherical and radial parts and matching the leading coefficients of both sides of \eqref{vacuum_constraint} give the following equations for $u^{(-2)}$ and ${Y_i}^{(-2)}$:
\[\begin{split}(\tilde{\Delta}+2) u^{(-2)}&=\frac{-1}{16}(\sum_k B_k\tilde{X}^k)^2-\frac{1}{4}\sum_i(B_i)^2\\
(\tilde{\Delta}+2) {Y_i}^{(-2)}&=-2\tilde{X}^i (\sum_k B_k\tilde{X}^k)-4B_i.\end{split}\]

The right hand sides of the equations are linear combination of constant functions and $-6$ eigenfunctions. General solutions of the equations
are
\[
\begin{split}
u^{(-2)}=&-\frac{9}{64}\sum_i (B_i)^2+\frac{1}{64} B_i B_j \tilde{X}^i \tilde{X}^j +c_i\tilde{X}^i\\
{Y_i}^{(-2)}=&-\frac{5}{2} B_i+\frac{1}{2}( B_k \tilde{X}^k)\tilde{X}^i+d_{ij} \tilde{X}^j,
\end{split}
\] where $c_i$ and $d_{ij}$ are constants.
\end{proof}
\begin{remark}
For harmonic asymptotic initial data sets, the energy-momentum, center of mass and angular momentum are determined by the coefficient $A$, $B_i$, $c_i$ and $d_{ij}$ respectively. The ADM energy-momentum vector is $(2A,\frac{B_i}{2})$ and the ADM angular momentum and BORT center of mass are
\begin{equation}\label{adm_bort}
\begin{split}
C^i_{BORT} = & 2c^i\\
J^i _{ADM}= & \frac{\epsilon^{ijk}d_{jk}}{2}.
\end{split}
\end{equation}
The exact coefficients in the above formula are from the calculations in \cite{Huang2}.
\end{remark}
Next, we compute the data  $(\sigma,|H|, \alpha_H)$ for coordinate spheres of harmonic asymptotic initial data sets. Sometimes it is more convenient
to express $g$ and $k$ in terms of  the spherical coordinate system associated with the asymptotically flat coordinate coordinate on each end.
Therefore 
\[g=u^4(dr^2+r^2\tilde{\sigma}_{ab} du^a du^b)\] and 
\[k= k_{rr}dr^2+2k_{ra} dr du^a+k_{ab} du^b du^b.\]

First we work with for any
 conformally flat metric.
\begin{proposition}\label{conformal_data_expansion}
Suppose an initial data set $(M, g, k)$ has conformally flat metric $g=u^4\delta$, then the geometric data on $\Sigma_r$ are
\[\begin{split} \sigma&=u^4 r^2\tilde{\sigma}\\
|H|&=\sqrt{\bar{h}^2-p^2}\\
(\alpha_H)_a&=\frac{\bar{h} \partial_a p-p\partial_a \bar{h}}{\bar{h}^2-p^2}- u^{-2}k_{ra},
\end{split}\] where 
\[ \bar{h}=2r^{-1} u^{-2}-2\partial_r(u^{-2}) \,\,\, {\rm and} \,\,\, p=u^{-4}k_{ij}(\delta^{ij}-\tilde{X}^i\tilde{X}^j).\]
\end{proposition}

\begin{proof}
The unit outward normal in the spatial direction is $e_3=u^{-2} \frac{\partial}{\partial r}$ and
the mean curvature of $\Sigma_r$ in $M$ is
\[\bar{h}=u^{-2}\partial_r \ln \sqrt{\det\sigma_{ab}},\] where
$\sqrt{\det\sigma_{ab}}=u^4 r^2\sqrt{\det {\tilde\sigma_{ab}}}$.

On the other hand, $p=\sigma^{ab}k_{ab}$ where $k_{ab}= r^2 k_{ij}(\partial_a \tilde X^i)(\partial_b \tilde X^j)$.  Thus 
\[p=u^{-4} r^{-2}\tilde{\sigma}^{ab} r^2 k_{ij}(\partial_a \tilde{X}^i)(\partial_b \tilde{X}^j)=u^{-4}k_{ij}(\delta^{ij}-\tilde{X}^i\tilde{X}^j).\]

Recall from \cite{Wang-Yau3},
\[(\alpha_H)_a= -k(e_3, \partial_a) + \partial_a \sinh^{-1} \frac{p}{\sqrt{\bar h^2 - p^2}}.\] 
Using $e_3=u^{-2} \frac{\partial}{\partial r}$ again, we conclude
\[(\alpha_H)_a=\frac{\bar{h} \partial_a p-p\partial_a \bar{h}}{\bar{h}^2-p^2}- u^{-2}k_{ra}.\] 
\end{proof}

\begin{proposition}
On a harmonic asymptotic initial data set with the expansion \eqref{expansion_u_Y}, the data on $\Sigma_r$ are
\[\begin{split} \sigma&=r^2\tilde{\sigma}_{ab}+4Ar\tilde{\sigma}_{ab}+(4u^{(-2)}+6A^2)\tilde{\sigma}_{ab} + o(1)\\
|H|&=\frac{2}{r}-8Ar^{-2} +r^{-3}[-12u^{(-2)}+18A^2-\frac{1}{4}(\sum_k B_k \tilde{X}^k)^2] + o(r^{-3})\\
(\alpha_H)_a&=\frac{3}{2}r^{-1}\partial_a(\sum_k B_k \tilde{X}^k)-r^{-2}[A\sum_k 6 B_k\partial_a \tilde{X}^k-\frac{5}{2}d_{ij} (\partial_a \tilde{X}^i)\tilde{X}^j+\frac{1}{2}d_{ij}\tilde{X}^i\partial_a \tilde{X}^j] + o(r^{-2}).
\end{split}\] 
\end{proposition}
\begin{proof}
The induced metric is $\sigma_{ab}=u^4r^2\tilde{\sigma}_{ab}$.  Therefore 
\[\sigma_{ab}=r^2\tilde{\sigma}_{ab}+4Ar\tilde{\sigma}_{ab}+(4u^{(-2)}+6A^2)\tilde{\sigma}_{ab}+ o(1).\]

We recall that from Proposition \ref{conformal_data_expansion}, 
$ |H| = \sqrt{\bar{h}^2-p^2} $
where
\[\begin{split}\bar{h}&=2r^{-1} u^{-2}-2\partial_r(u^{-2})\\
p&=u^{-4}k_{ij}(\delta^{ij}-\tilde{X}^i\tilde{X}^j).\end{split}\]
For a harmonic asymptotic initial data set, we have
\[\bar{h}=\frac{2}{r}-r^{-2}(8A)+r^{-3}(-12u^{(-2)}+18A^2) + o(r^{-3}).\]
Therefore, $\bar{h}^{(-2)}= -8A$ and $\bar{h}^{(-3)}=-12u^{(-2)}+18A^2$. Let
\[p=r^{-2} p^{(-2)}+r^{-3}p^{(-3)}+o(r^{-3}).\]
We have
\[
\begin{split}
Y_{i, j}&=r^{-2}(-B_i \tilde{X}^j)+r^{-3}[-2 Y_i^{(-2)}\tilde{X}^j+\tilde{\nabla}Y_i^{(-2)} \cdot \tilde{\nabla}\tilde{X}^j] + o(r^{-3}).
\end{split}
\]
and
\[\begin{split}p&=-u^{-2}(2 Y_{i, j}\tilde{X}^i \tilde{X}^j-\sum_k Y_{k, k})+o(r^{-3})\\
&=-u^{-2}[r^{-2}(-\sum_k (B_k \tilde{X}^k)+r^{-3}(-2Y_i^{(-2)}\tilde{X}^i-\tilde{\nabla}Y_i^{(-2)}\tilde{\nabla}\tilde{X}^i)]+ o(r^{-3})\\
&=r^{-2}\sum_k (B_k \tilde{X}^k)+r^{-3}[2Y_i^{(-2)}\tilde{X}^i+\tilde{\nabla}Y_i^{(-2)}\tilde{\nabla}\tilde{X}^i-2A\sum_k (B_k \tilde{X}^k)]+ o(r^{-3}).\end{split}\]

Therefore $p^{(-2)}=\sum_k (B_k \tilde{x}^k)$ and 
\[p^{(-3)}=2Y_i^{(-2)}\tilde{X}^i+\tilde{\nabla}Y_i^{(-2)}\tilde{\nabla}\tilde{X}^i-2A\sum_k (B_k \tilde{X}^k).\]
We obtain 
\[
\begin{split}
|H|=&\frac{2}{r}+r^{-2} \bar{h}^{(-2)}+r^{-3}[\bar{h}^{(-3)}-\frac{1}{4}(p^{(-2)})^2]+ o(r^{-3}) \\
=&\frac{2}{r}-8Ar^{-2} +r^{-3}[-12u^{(-2)}+18A^2-\frac{1}{4}(\sum_k B_k \tilde{x}^k)^2]+ o(r^{-3}).
\end{split}
\]

Finally, we compute $\alpha_H$. In general, with \[
\begin{split}
\bar{h}= &\frac{2}{r}+r^{-2}\bar{h}^{(-2)}+r^{-3}\bar{h}^{(-3)} + o(r^{-3})\\ 
p=&r^{-2} p^{(-2)}+r^{-3}p^{(-3)}+o(r^{-3}),
\end{split}
\] we have 
\[(\alpha_H)_a=\frac{\bar{h} \partial_a p-p\partial_a \bar{h}}{\bar{h}^2-p^2}- u^{-2}k_{ra}. \]

For harmonic asymptotic initial data sets, $\bar{h}^{(-2)}=-8A$, $p^{(-2)}=\sum_k (B_k \tilde{X}^k) 
$ and 
\[ p^{(-3)}=2Y_i^{(-2)}\tilde{X}^i+\tilde{\nabla}Y_i^{(-2)}\cdot \tilde{\nabla}\tilde{x}^i-2A\sum_k (B_k \tilde{X}^k). \] 
Hence,
\[
\begin{split}
\frac{\bar{h} \partial_a p-p\partial_a \bar{h}}{\bar{h}^2-p^2}=&\frac{1}{2}r^{-1}\partial_a p^{(-2)}+\frac{1}{4}r^{-2}[ 2\partial_a p^{(-3)}-h^{(-2)}\partial_a p^{(-2)}-\partial_a \bar{h}^{(-2)}p^{(-2)}] +o(r^{-2})\\
=&\frac{1}{2}r^{-1}\partial_a p^{(-2)}+r^{-2}( \frac{1}{2}\partial_a p^{(-3)}+2A \partial_a p^{(-2)})+o(r^{-2}) \\
 u^{-2} k_{ra}=& r^{-1}(-\sum_j B_j \partial_a \tilde{X}^j)+r^{-2}(\tilde{X}^i \partial_a Y_i^{(-2)}-2Y_i^{(-2)}\partial_a\tilde{X}^i)+o(r^{-2})
\end{split}
\]
and
\[(\alpha_H)_a=\frac{3}{2}r^{-1}\partial_a(\sum_k B_k \tilde{X}^k)+r^{-2}[A\sum_k \partial_a (B_k \tilde{X}^k)+3
Y_i^{(-2)}\partial_a\tilde{X}^i+\frac{1}{2}\partial_a( \tilde{\nabla}Y_i^{(-2)}\cdot \tilde{\nabla}\tilde{X}^i)] + o(r^{-3}).\]

Using Lemma \ref{harmonic_expansion}, we have \[3Y_i^{(-2)}\partial_a\tilde{X}^i+\frac{1}{2}\partial_a( \tilde{\nabla}Y_i^{(-2)}\cdot \tilde{\nabla}\tilde{X}^i)=-7 B_k\partial_a \tilde{X}^k+\frac{5}{2}d_{ij} (\partial_a \tilde{X}^i)\tilde{X}^j-\frac{1}{2}d_{ij}\tilde{X}^i\partial_a \tilde{X}^j.\]
\end{proof}

\section{CWY total conserved quantities for  harmonic asymptotic initial data}
In this section, we evaluate the CWY total conserved quantities for harmonic asymptotic initial data sets. 
\begin{lemma} \label{HA_optimal_expansion}
For a harmonic asymptotic initial data set with the expansion \eqref{expansion_u_Y}, the solution of the optimal isometric embedding equation \eqref{optimal} has the following
expansion:
\[
\begin{split}
X^i = &  r \tilde X^i + 2A \tilde{X}^i +r^{-1}(2u^{(-2)}+A^2)\tilde{X}^i + o(r^{-1}) \\ 
X^ 0 = & r^{-1}(X^0)^{(-1)}+ o(r^{-1})\\
T_0 = &(a^0,a^i) + r^{-1}(a_0^{(-1)},a_i^{(-1)})+o(r^{-1})
\end{split}
\]
where 
\[\frac{a^k}{a^0}=\frac{B_k}{4A}.\]
\end{lemma}
\begin{remark}
The terms $(X^0)^{(-1)}$ and $a_0^{(-1)}, a_i^{(-1)}, i=1, 2, 3$  can be solved from the second order terms of the optimal isometric embedding equation. However, the total conserved quantities do not depend on them.
\end{remark}
\begin{proof}
First we solve $(X^i)^{(0)}$ from equation \eqref{isometric1}. For a harmonic asymptotic initial data, the equation reads
\[
\partial_a \tilde{X}^i\partial_b (X^i)^{(0)}+\partial_a (X^i)^{(0)}\partial_b \tilde{X}^i=4A \tilde \sigma_{ab}
\]
and we have $ (X^i)^{(0)}=2A \tilde{X}^i$. Since  $h_0^{(-2)}$ is given by
\[ h_0^{(-2)}=-\tilde{X}^i \tilde{\Delta}(X^i)^{(0)}-\tilde{\sigma}^{ab}\sigma_{ab}^{(1)}, \]  we compute $h_0^{(-2)}=-4A$. On the other hand, $h^{(-2)}=-8A$. 
We also have 
\[(\alpha_H)^{(-1)}=\frac{3}{2}\partial_a (B^k\tilde{X}^k).\]

Next we  solve $(X^0)^{(0)}$ from  equation \eqref{linearized_optimal}.  With harmonic asymptotics, we have  $\tau^{(1)} =-\sum_{i=1}^{3} a_{i} \tilde X^i$ and
\[ 
\rho^{(-2)}=\frac{h_0^{(-2)}-h^{(-2)}}{a^0}=\frac{4A}{a^0}.
\]

Equation  \eqref{linearized_optimal} becomes

\[3(\frac{4A}{a^0} a^k-B^k)\tilde{X}^k-\frac{1}{2}\widetilde{\Delta}(\widetilde{\Delta}+2)(X^0)^{(0)}=0.\] The solvability condition demands that
 $\frac{a^k}{a^0}=\frac{B_k}{4A}$ and $(X^0)^{(0)}$ can be chosen to be zero. 

Next, we solve $(X^i)^{(-1)}$  from equation \eqref{isometric2}. With harmonic asymptotics, the equation becomes

\begin{equation*}\partial_a \tilde{X}^i\partial_b (X^i)^{(-1)}+\partial_a (X^i)^{(-1)}\partial_b \tilde{X}^i= 
\partial_a (X^0)^{(0)}\partial_b (X^0)^{(0)}+(4u^{(-2)}+2A^2)\tilde{\sigma}_{ab}\end{equation*}

With $(X^0)^{(0)}=0$, we can take $(X^i)^{(-1)}=(2u^{(-2)}+A^2)\tilde{X}^i$. 
\end{proof}

We have the following expansion for the data on  $X_r(\Sigma_r)$ obtained in Lemma \ref{HA_optimal_expansion}.

\begin{lemma}
For the solution of the optimal embedding equation obtained in  Lemma \ref{HA_optimal_expansion}, 
\[
\begin{split}
|H_0| = & \frac{2}{r}-4Ar^{-2}+r^{-3}(-2\tilde{\Delta} u^{(-2)}-4u^{(-2)}+6A^2) + o(r^{-3}) \\
(\alpha_{H_0})_a= & r^{-2}[\partial_a (X^0)^{(-1)}+\frac{1}{2}\partial_a(\tilde{\Delta} (X^0)^{(-1)}] + o(r^{-2}).
\end{split}
\]
\end{lemma}
\begin{proof}
We compute $H_0= (\Delta X^0,\Delta X^i )$,
where $\Delta X^0=\frac{\tilde \Delta  (X^0)^{(-1)}}{r^3}$ and
\[
\begin{split}
\Delta X^i= & r^{-1}\tilde{\Delta} (X^i)^{(1)}+ r^{-2}  [\tilde{\Delta} (X^i)^{(0)}+ \Delta^{(-3)}(X^i)^{(1)} ] \\
&+r^{-3}  [\tilde{\Delta} (X^i)^{(-1)}+ \Delta^{(-4)}(X^i)^{(1)}  + \Delta^{(-3)}(X^i)^{(0)} ] +o(r^{-3}).
\end{split}
\]

For harmonic asymptotic initial data sets,  $\sigma_{ab}=u^4 r^2\tilde{\sigma}_{ab}$ and
$\Delta=u^{-4} r^{-2}\tilde{\Delta}$. Therefore 
\[
\begin{split}
\Delta^{(-3)}=&-4A\tilde{\Delta}\\
\Delta^{(-4)}=&(10A^2-4u^{(-2)})\tilde{\Delta}.
\end{split}
\]
Using the expansion of $X^i$ from Lemma \ref{HA_optimal_expansion}, we have
\[|H_0|=\frac{2}{r}-4Ar^{-2}+r^{-3}(-2\tilde{\Delta} u^{(-2)}-4u^{(-2)}+6A^2).\]
The expansion for $(\alpha_{H_0})_a$ follows from $(X^0)^{(0)}=0$.
\end{proof}

Next, we compute the expansion for $\rho$ and $j_a$ in the following lemmas.
\begin{lemma} \label{HA_expansion_rho}
For the solution of the optimal embedding equation obtained in  Lemma \ref{HA_optimal_expansion}, we have $\rho = r^{-2} \rho^{(-2)} +r^{-3} \rho^{(-3)} + o(r^{-3})$ where $ \rho^{(-2)}  = \frac{4A}{a^0}$ and
\[  a_0^3 \rho^{(-3)}=  a_0^2(-\frac{19}{16}\sum_i B_i^2+\frac{9}{16}(\sum B_i\tilde{X}^i)^2+12 c_i\tilde{X}^i-12A^2)-4A^2(\sum_i a_i \tilde{X}^i)^2-4A\sum_i a_i a_i^{(-1)} .  \]
\end{lemma}

\begin{proof}
With $\tau=-\langle X, T_0\rangle=r(-a_i \tilde{X}^i)+(-2Aa_i - a_i^{(-1)})\tilde{X}^i + o(1)$, we compute
\[
\begin{split}
(\Delta\tau)^2= & r^{-2} 4(\sum_i a_i \tilde{X}^i)^2+r^{-3}(-16 A a_i a_j \tilde{X}^i\tilde{X}^j+8a_i a_j^{(-1)}\tilde{X}^i\tilde{X}^j) +o(r^{-3})\\
1+|\nabla\tau|^2= &a_0^2-(\sum_i a_i \tilde{X}^i)^2+r^{-1}(2a_i a_i^{(-1)}-2a_ia_j^{(-1)}\tilde{X}^i\tilde{X}^j)+o(r^{-1}).
\end{split}
\]
It follows that 
\[\begin{split}\rho
&=\frac{4A}{a_0} r^{-2}+r^{-3} a_0^{-3}[a_0^2(h_0^{(-3)}-h^{(-3)})-4A^2(\sum_i a_i \tilde{X}^i)^2-4A\sum_i a_i a_i^{(-1)}]+o(r^{-3}).\end{split}\]
By a direct computation, we obtain
\[h_0^{(-3)}-h^{(-3)}=-\frac{19}{16}\sum_i B_i^2+\frac{9}{16}(\sum B_i\tilde{X}^i)^2+12 c_i\tilde{X}^i-12A^2.\]
\end{proof}

\begin{lemma}
For the solution of the optimal embedding equation obtained in  Lemma \ref{HA_optimal_expansion}. 
Set $g= B_i \tilde{X}^i$. 
We have $ j_a=r^{-1}j_a^{(-1)}+r^{-2}j_a^{(-2)}  + o(r^{-2}),$ where
\[\begin{split}
j^{(-1)}_a =&-\frac{3}{2}\partial_a g + (\alpha_H^{(-1)})_a \\
 j^{(-2)}_a 
= &
\frac{1}{2}A\partial_a g-\frac{6A}{a_0}a_i^{(-1)}\partial_a\tilde{X}^i+\frac{57}{128A} (\sum_i B_i^2)\partial_a g-\frac{25}{128A}g^2\partial_a g
-\frac{9}{2A}(c_i\tilde{X}^i)\partial_a g\\
&+\frac{3}{8Aa_0}(\sum_i B_ia_i^{(-1)})\partial_a g-\frac{3}{2A} g(c_i\partial_a \tilde{X}^i) - (\alpha_{H_0}^{(-2)})_a    + (\alpha_H^{(-2)} )_a \end{split}\]
\end{lemma}
\begin{proof}
Recall that
\[j_a=\rho \partial_a \tau - \partial_a [ \sinh^{-1} (\frac{\rho \Delta \tau }{|H_0||H|})] - (\alpha_{H_0})_a + (\alpha_{H})_a.\]
We compute 
\[\begin{split}\sinh^{-1}(\frac{\rho\Delta\tau}{|H| |H_0|}) = &\frac{1}{4}r^{-1}\rho^{(-2)}(\Delta\tau)^{(-1)} \\
&+\frac{1}{4} r^{-2}[\rho^{(-2)}(\Delta\tau)^{(-2)}+\rho^{(-3)}(\Delta\tau)^{(-1)}-\frac{1}{2}(h^{(-2)}+h_0^{(-2)})\rho^{(-2)}(\Delta\tau)^{(-1)}]+ o(r^{-2})\end{split}\]

Since $\Delta\tau=r^{-1} (2a_i\tilde{x}^i)+r^{-2}[2(a_i^{(-1)}-2A a_i)\tilde{x}^i]$ and $\rho^{(-2)}=\frac{4A}{a_0}$,  we obtain
\[
\begin{split}
\partial_a \sinh^{-1}(\frac{\rho\Delta\tau}{|H| |H_0|})=& r^{-1}\partial_a(\frac{2A}{a_0} a_i\tilde{X}^i)+r^{-2}\partial_a(\frac{2A}{a_0}a_i^{(-1)} \tilde{X}^i+\frac{8A^2}{a_0} a_i \tilde{X}^i+\frac{1}{2}
\rho^{(-3)} a_i\tilde{X}^i) + o(r^{-2})\\
\rho\partial_a \tau= &-r^{-1}(\frac{4A}{a_0}a_i\partial_a \tilde{X}^i)-r^{-2}[\frac{4A}{a_0}(2Aa_i+a_i^{(-1)})+\rho^{(-3)}a_i]\partial_a \tilde{X}^i + o(r^{-2}).
\end{split}
\]
Collecting the terms gives the desired expansion for the lemma.
\end{proof}

Recall that the conserved quantities are the limit of 

\[
\begin{split}&E(\Sigma_r, X, T_0, K)=\frac{(-1)}{8\pi} \int_\Sigma
\left[ \langle K, T_0\rangle \rho+j(K^\top) \right]d\Sigma,\end{split}
\]
where $K$ is a Killing vector field in $\R^{3,1}$ of the form
\[  K= K_{\alpha}^{\,\,\, \beta} x^\alpha \frac{\partial}{\partial x^\beta} \] with $K_{\alpha \beta}$ being skew-symmetric. We first compute the conserved quantities for such $K$ in general and then specify the choice of $K_{\alpha \beta}$ corresponding to angular momentum and center of mass later. 
\begin{lemma}\label{evaluation_1}
We have
\[\begin{split}\int_{\Sigma_r} \langle K, T_0\rangle \rho \sqrt{\sigma}=& \int_{S^2}   [ K_{\alpha\gamma} (X^\alpha)^{(1)} (T^\gamma)^{(0)} \rho^{(-3)}]  d S^2
+o(1)  \\
 \int_{\Sigma_r} \langle K, \frac{\partial X}{\partial u^a} \rangle \sigma^{ab} j_b\sqrt{\sigma}=& \int_{S^2} [K_{\alpha\gamma} (X^\alpha)^{(1)} \frac{\partial (X^\gamma)^{(1)}}{\partial u^a}j_b^{(-2)}]
\tilde{\sigma}^{ab}  d S^2+o(1).
\end{split}\]

\end{lemma}
\begin{proof}
\[\begin{split} & \int_{\Sigma_r} \langle K, T_0\rangle \rho \sqrt{\sigma} \\
= &\int_{S^2}  (r \langle K, T_0\rangle^{(1)}+ \langle K, T_0\rangle^{(0)})(r^{-2} \rho^{(-2)}+r^{-3} \rho^{(-3)})(r^2\sqrt{\tilde{\sigma}}+r\sqrt{\sigma}^{(1)})+o(1)\\
= &\int_{S^2}   [ \langle K, T_0\rangle^{(1)} \rho^{(-3)}+\langle K, T_0\rangle^{(0)} \rho^{(-2)}
+\frac{1}{2}\tilde{\sigma}^{ab}\sigma_{ab}^{(1)} \langle K, T_0\rangle^{(1)} \rho^{(-2)}]\sqrt{\tilde{\sigma}}+o(1).\\\end{split}\]
On the other hand,
\[\begin{split}
\sigma^{ab}\sqrt{\sigma}=&\tilde{\sigma}^{ab}\sqrt{\tilde{\sigma}}+r^{-1}[\frac{1}{2}\tilde{\sigma}^{ab}(\tilde{\sigma}^{pq}\sigma_{pq}^{(1)}) 
-\tilde{\sigma}^{ap}\sigma_{pq}^{(1)}\tilde{\sigma}^{qb}]\sqrt{\tilde{\sigma}} +o(r^{-1})\\
=&\tilde{\sigma}^{ab}\sqrt{\tilde{\sigma}}+o(r^{-1})
\end{split}\]
since $\sigma_{ab}^{(1)}=\frac{1}{2}tr_{\tilde{\sigma}}\sigma^{(1)}\tilde{\sigma}_{ab}$. Therefore, 
\[\begin{split}& \int_{\Sigma_r} \langle K, \frac{\partial X}{\partial u^a} \rangle \sigma^{ab} j_b\sqrt{\sigma}\\
=&\int_{S^2} (r^2 \langle K, \frac{\partial X}{\partial u^a} \rangle^{(2)}+r \langle K, \frac{\partial X}{\partial u^a} \rangle^{(1)})  \tilde{\sigma}^{ab}(r^{-1} j_b^{(-1)}+r^{-2} j_b^{(-2)})\sqrt{\tilde{\sigma}} + o(1).
\end{split}\]
We have
\[
\begin{split}\int_{\Sigma_r} \langle K, T_0\rangle \rho \sqrt{\sigma}
=& \int_{S^2}   [ \langle K, T_0\rangle^{(1)} \rho^{(-3)}+\langle K, T_0\rangle^{(0)} \rho^{(-2)}] \sqrt{\tilde{\sigma}}
+o(1)\\
 \int_{\Sigma_r} \langle K, \frac{\partial X}{\partial u^a} \rangle \sigma^{ab} j_b\sqrt{\sigma}= &\int_{S^2} [\langle K, \frac{\partial X}{\partial u^a} \rangle^{(2)}j_b^{(-2)}+ \langle K, \frac{\partial X}{\partial u^a} \rangle^{(1)}j_b^{(-1)}]
\tilde{\sigma}^{ab} \sqrt{\tilde{\sigma}}+o(1)
\end{split}\]
since   $X^{(0)}=2A X^{(1)}$. Moreover,
\[
\begin{split}
\langle K, T\rangle = & r K_{\alpha\gamma} (X^\alpha)^{(1)} (T^\gamma)^{(0)}+K_{\alpha\gamma} [(X^\alpha)^{(1)} (T^\gamma)^{(-1)}
+(X^\alpha)^{(0)} (T^\gamma)^{(0)}] + o(1) \\
\langle K,  \frac{\partial X}{\partial u^a} \rangle=& r^2 K_{\alpha\gamma} (X^\alpha)^{(1)} \frac{\partial (X^\gamma)^{(1)}}{\partial u^a}
+rK_{\alpha\gamma} [(X^\alpha)^{(1)} \frac{\partial (X^\gamma)^{(0)}}{\partial u^a}+(X^\alpha)^{(0)} \frac{\partial (X^\gamma)^{(1)}}{\partial u^a}] + o(r).
\end{split}\]

Hence,
\[\begin{split}\int_{\Sigma_r} \langle K, T_0\rangle \rho \sqrt{\sigma}=&\int_{S^2}   [ K_{\alpha\gamma} (X^\alpha)^{(1)} (T^\gamma)^{(0)} \rho^{(-3)}] \sqrt{\tilde{\sigma}}
+o(1)\\
 \int \langle K, \frac{\partial X}{\partial u^a} \rangle \sigma^{ab} j_b\sqrt{\sigma}=&\int_{S^2} [K_{\alpha\gamma} (X^\alpha)^{(1)} \frac{\partial (X^\gamma)^{(1)}}{\partial u^a}j_b^{(-2)}]
\tilde{\sigma}^{ab} \sqrt{\tilde{\sigma}}+o(1),
\end{split}\]
since $K_{\alpha\gamma} [(X^\alpha)^{(1)} \frac{\partial (X^\gamma)^{(0)}}{\partial u^a}+(X^\alpha)^{(0)} \frac{\partial (X^\gamma)^{(1)}}{\partial u^a}]=0$ and $X^{(0)}=2A X^{(1)}$. 

\end{proof}
The following computational lemma is useful in the evaluation of CWY total conserved quantities. 
\begin{lemma}\label{evaluation_2}
We have 
\[
\begin{split}
\frac{1}{8 \pi}\int_{S^2} \tilde X^k \rho^{(-3)} dS^2=&  \frac{2c^k }{a^0}\\
\frac{1}{8 \pi}\int_{S^2} \tilde X^k \tilde X^l_a \tilde \sigma^{ab}  j_b^{(-2)}dS^2=&\frac{c_l B_k - c_kB_ l}{4A}+\frac{d_{lk}-d_{kl}}{4}.\\
\end{split}
\]
\end{lemma}
\begin{proof}
Recall that 
\[ \int_{S^2} (\tilde X^1)^a (\tilde X^2)^b (\tilde X^3)^c dS^2= 0  \]
unless $a$, $b$ and $c$ are all even. On the other hand,
\[  \int_{S^2} \tilde X^i \tilde X^j dS^2 = \frac{4 \pi \delta^{ij}}{3}.   \]
Using the above formula and the expression of $\rho^{(-3)}$ from Lemma \ref{HA_expansion_rho}, we get 
\[ \frac{1}{8 \pi}\int_{S^2} \tilde X^k \rho^{(-3)} dS^2=\frac{2c^k}{a^0}.  \]
For the second identity, recall that from the second order term of the optimal embedding equation, $j_{b}^{(-2)}$ is divergence free. Hence,
\[
\begin{split}
\int_{S^2} \tilde X^k \tilde X^l_a \tilde \sigma^{ab}  j_b^{(-2)}dS^2 
=  \frac{1}{2} \int_{S^2} (\tilde X^k \tilde X^l_a-\tilde X^l \tilde X^k_a ) \tilde \sigma^{ab}  j_b^{(-2)}dS^2. 
\end{split}
\]
Moreover,
\[
\begin{split}
 & \int_{S^2} (\tilde X^k \tilde X^l_a-\tilde X^l \tilde X^k_a )\tilde \sigma^{ab} [\partial_b (X^0)^{(-1)}+\frac{1}{2}\partial_b(\tilde{\Delta} (X^0)^{(-1)}] dS^2  \\
= &- \int_{S^2} \tilde \nabla^a(\tilde X^k \tilde X^l_a-\tilde X^l \tilde X^k_a )\tilde  [ (X^0)^{(-1)}+\frac{1}{2}(\tilde{\Delta} (X^0)^{(-1)}] dS^2
=  0
\end{split}
\]
and
\[
\begin{split}
  &  \int_{S^2} \tilde X^k \tilde X^l_a \tilde \sigma^{ab}  j_b^{(-2)}dS^2  \\
= &\frac{1}{2}  \int_{S^2} (\tilde X^k \tilde X^l_a-\tilde X^l \tilde X^k_a ) \tilde \sigma^{ab} \Big [  \frac{1}{2}A\partial_b g-\frac{6A}{a_0}a_i^{(-1)}\partial_b\tilde{X}^i+\frac{57}{128A} (\sum_i B_i^2)\partial_b g \\
&-\frac{25}{128A}g^2\partial_b g -\frac{9}{2A}(c_i\tilde{X}^i)\partial_b g
+\frac{3}{8Aa_0}(\sum_i B_ia_i^{(-1)})\partial_b g-\frac{3}{2A} g(c_i\partial_b \tilde{X}^i)    \\
& -A\sum_k 6 B_k\partial_b \tilde{X}^k+\frac{5}{2}d_{ij} (\partial_b \tilde{X}^i)\tilde{X}^j-\frac{1}{2}d_{ij}\tilde{X}^i\partial_b \tilde X^j \Big ]   dS^2 \\
=& \frac{1}{2} \int_{S^2} (\tilde X^k \tilde X^l_a-\tilde X^l \tilde X^k_a ) \tilde \sigma^{ab} \Big [
\frac{-3}{A}(c_i\tilde{X}^i)\partial_b B_j \tilde x^j +\frac{3}{2}d_{ij} [\partial_b (\tilde{X}^i)\tilde{X}^j- (\partial_b \tilde{X}^j)\tilde{X}^i ]
\Big ] dS^2 \\
=& \frac{3}{2} \int_{S^2} \frac{1}{A}[\tilde X^l \tilde X^i  c_iB_j \delta^{kj} - \tilde X^k \tilde X^i  c_iB_j \delta^{lj}  ] +  [d_{ij} \tilde X^j \tilde X^k \delta^{li} - d_{ij} \tilde X^j \tilde X^l  \delta^{kj}  ] dS^2\\
= & \frac{2\pi(c_lB_ k- c_k B_l)}{A}+2(d_{lk}-d_{kl})\pi.
\end{split}
\]
\end{proof}
\begin{theorem}\label{thm_final}
For a harmonic asymptotic initial data set, let $X(r),T_0(r)$ be the solution of optimal embedding equation on $\Sigma_r$ obtained in Lemma \ref{HA_optimal_expansion}. Let  $ T_0(r)= A(r) \frac{\partial}{\partial x^0}$
where
\[ A(r) = A_{\alpha \beta} + O(r^{-1}).  \]
 The CWY total angular momentum and center of mass are 
\[
\begin{split}
C_i = & \frac{2c_i}{a^0} + \frac{(c_i B_j - c_j B_i)A_{0j}}{4A}+\frac{(d_{ij}-d_{ji})A_{0j}}{4}\\
J_i= & \sum_{j,k,l}\frac{ \epsilon_{ijk}}{4}[\frac{A_{kl} ( c_jB_l  - c_l B_j)}{A}+ A_{kl}(d_{jl} - d_{lj})].
\end{split}
\]
\end{theorem}
\begin{proof}
Recall that the definitions from Definition \ref{total_conserved} and we compute
\[
\begin{split}
 & \frac{1}{8 \pi}E(\Sigma_r, X(r), T_0(r), A({r})(x^i\frac{\partial}{\partial x^0}+x^0\frac{\partial}{\partial x^i}))\\
=& \frac{1}{8 \pi} \int_{\Sigma_r} X^i \rho - (X^i A_{0 \alpha}(r) X^{\alpha}_a +X^0 A_{i \alpha}(r)X^\alpha_a ) \sigma^{ab} j_{b} d \Sigma_r \\
= & \frac{1}{8 \pi} \int_{S^2} \tilde X^i \rho^{(-3)} - ( \tilde X^i A_{0 j}  \tilde X^{j}_a  ) \tilde \sigma^{ab} j_{b}^{(-2)} dS^2 + o(1) \\
=& \frac{2c_i}{a^0} + \frac{(c_i B_j - c_j B_i)A_{0j}}{4A}+ \frac{( d_{ij}- d_{ji})A_{0j}}{4}+ o(1).
\end{split}
\]
Similarly,
\[
\begin{split}
 & \frac{1}{8 \pi}E(\Sigma_r , X(r), T_0(r), A({r})(x^j\frac{\partial}{\partial x^k}-x^k\frac{\partial}{\partial x^j})) \\
=& \frac{-1}{8 \pi}\int_{\Sigma_r} (X^j A_{ k \alpha}(r) X^{\alpha}_a +X^j A_{k \alpha}(r)X^\alpha_a ) \sigma^{ab} j_{b} d \Sigma_r \\
= & \frac{-1}{8 \pi} \int_{S^2}  (\tilde X^j A_{ k l} \tilde  X^{l}_a + \tilde X^j A_{k l}\tilde X^l_a )  \tilde \sigma^{ab} j_{b}^{(-2)} dS^2 + o(1) \\
=& \frac{A_{kl} ( c_jB_l  - c_l B_j)}{4A}+\frac{ A_{kl}(d_{jl} - d_{lj})}{4}+ o(1).
\end{split}
\]
\end{proof}   
\begin{remark} By comparing \eqref{adm_bort} with the above theorem, we conclude that
for a harmonic asymptotic initial data set with vanishing linear momentum ($B_i=0$ and $A_{\alpha \beta} = \delta_{\alpha \beta}$), the CWY total angular momentum and center of mass are the same as the ADM angular momentum and BORT center of mass, respectively. When the linear momentum is non-zero, the CWY total angular momentum and center of mass become linear combinations of the ADM angular momentum and BORT center of mass, where the coefficients depend on the ADM energy-momentum vector.
\end{remark}

\end{document}